\documentclass[11pt]{amsart}

	\usepackage{amsfonts,amsmath, amsthm, amssymb, latexsym, epsfig}
	\usepackage[all]{xy}
	\usepackage{xspace}
	\usepackage{comment}
	\usepackage{setspace}
	\usepackage{xcolor}
	\usepackage{enumerate}
	\usepackage[mathscr]{eucal}
	\usepackage{todonotes}
	\usepackage{bbold}
	\advance \topmargin by -\headheight
	\advance \topmargin by -\headsep
	\evensidemargin \oddsidemargin
	\newcommand{\ftn}[3]{ #1 \colon #2 \rightarrow #3 }
	
	\newcommand{\multialg}[1]{\mathcal{M}(#1)}
	\newcommand{\corona}[1]{\mathcal{Q}(#1)}
	
	\newcommand{\C}{\mathbb{C}}

	\newcommand{\N}{\mathbb{N}}
	\newcommand{\K}{\mathbb{K}}
		\newcommand{\I}{\mathfrak{I}}
	\newcommand{\E}{\mathfrak{E}}
	\newcommand{\F}{\mathfrak{F}}
	\newcommand{\A}{\mathfrak{A}}
\renewcommand{\phi}{\varphi}

	\newcommand{\kk}{\operatorname{KK}}
	
	\newcommand{\id}{\operatorname{id}}

	\theoremstyle{plain}
	\newtheorem{thm}{Theorem}

	\newtheorem{proposition}[thm]{Proposition}

	\theoremstyle{definition}

	\newtheorem{example}[thm]{Example}

		\newtheorem{observation}[thm]{Observation}

	\numberwithin{equation}{section}
	\numberwithin{figure}{section}

\begin{document}
\author{S{\o}ren Eilers}
        \address{Department of Mathematical Sciences \\
        University of Copenhagen\\
        Universitetsparken~5 \\
        DK-2100 Copenhagen, Denmark}
        \email{eilers@math.ku.dk }
        \author{Gunnar Restorff}
\address{Faculty of Science and Technology\\University of Faroe 
Islands\\N\'oat\'un 3\\FO-100 T\'orshavn\\Faroe Islands}
\email{gunnarr@setur.fo}

	\author{Efren Ruiz}
        \address{Department of Mathematics\\University of Hawaii,
Hilo\\200 W. Kawili St.\\
Hilo, Hawaii\\
96720-4091 USA}
        \email{ruize@hawaii.edu}
        \date{\today}
	
%AMS info


\title{Corrigendum to "Classifying $C^{*}$-algebras with both finite and infinite subquotients" }

\begin{abstract}
As recently pointed out by Gabe, a fundamental paper by Elliott and Ku\-ce\-rov\-sky concerning the absorption theory for $C^*$-algebras contains an error, and as a consequence we must report that 
Lemma~4.5 in \cite{segrer:ccfis} is not true as stated.   In this corrigendum, we prove an adjusted statement and explain why the error has no consequences to the main results of \cite{segrer:ccfis}. In particular, it is noted that all the authors' claims concerning Morita equivalence or stable isomorphism of graph $C^*$-algebras remain correct as stated.
\end{abstract}

\maketitle

In this note, we give a counterexample to \cite[Lemma~4.5]{segrer:ccfis} and we make the necessary changes to make the statement true.  Before doing this, we first explain where the error occurred.  In the proof of \cite[Lemma~4.5]{segrer:ccfis} we used \cite[Corollary~16]{EK:purelylarge} to conclude that a non-unital, purely large extension is nuclear absorbing.  This was the key component to prove \cite[Lemma~4.5]{segrer:ccfis}.  However, it was recently pointed out by James Gabe in \cite{gabe:absorb} that \cite[Corollary~16]{EK:purelylarge} is false in general;  Gabe showed that there exists a non-unital extension that is purely large but not nuclear absorbing.  The error occurs for non-unital extensions $0 \to \I \to \E \to \A \to 0$ with $\A$ unital.  We can use \cite[Example~1.1]{gabe:absorb}, to find a counterexample to \cite[Lemma~4.5]{segrer:ccfis} as follows:

\begin{example}\label{ex: counter example}
Let $p$ be a projection in $\mathbb{B}( \ell^{2} )$ such that $p$ and $1_{ \mathbb{B}( \ell^{2} ) } - p$ are norm-full, properly infinite projections in $\mathbb{B}( \ell^{2} )$.  Let $
\mathfrak{e} \colon 0 \to \K \oplus \K \to \E \to \C \to 0$ be the trivial extension induced by the $*$-homomorphism which maps $\lambda \in \C$ to $\lambda  ( p \oplus 1_{ \mathbb{B}( \ell^{2} ) } )$.  Since $p$ and $1_{ \mathbb{B}( \ell^{2} ) } - p$ are norm-full, properly infinite projections in $\mathbb{B}( \ell^{2} )$, we have that $p$ and $1_{ \mathbb{B}( \ell^{2} ) } - p$ are not elements of $\K$.  Therefore, $1_{ \mathbb{B}( \ell^{2} ) } \oplus 1_{ \mathbb{B}( \ell^{2} ) } - p \oplus 1_{ \mathbb{B}( \ell^{2} ) } = ( 1_ {\mathbb{B}( \ell^2 )} - p ) \oplus 0$ is not an element of $\K \oplus \K$.  Hence, $\mathfrak{e}$ is a non-unital extension.  By \cite[Example~1.1]{gabe:absorb}, $\mathfrak{e}$ is a purely large, full extension that is not nuclear absorbing.  Therefore, $\mathfrak{e}$ is not absorbing since $\C$ is a nuclear $C^*$-algebra.  Therefore, $\mathfrak{e}$ can not be isomorphic to an absorbing extension. 

We now construct a non-unital, absorbing extension $\mathfrak{f} \colon 0 \to \K \oplus \K \to \F \to \C \to 0$ such that $[ \tau_{ \mathfrak{e} } ] = [ \tau_{ \mathfrak{f} } ]$ in $\kk^{1} ( \C , \K \oplus \K )$, where $\tau_{\mathfrak{e}}$ and $\tau_{\mathfrak{f}}$ are the Busby invariants of $\mathfrak{e}$ and $\mathfrak{f}$ respectively.  Let $q$ be a projection in $\mathbb{B}( \ell^{2} )$ such that $q$ and $1_{ \mathbb{B}( \ell^{2} ) } - q$ are norm-full, properly infinite projections in $\mathbb{B}( \ell^{2} )$.  Let $\mathfrak{f} \colon 0 \to \K \oplus \K \to \F \to \C \to 0$ be the trivial extension induced by the $*$-homomorphism which maps $\lambda \in \C$ to $\lambda  ( p \oplus q )$.  Using a similar argument as in the case for $\mathfrak{e}$, we have that $\mathfrak{f}$ is a non-unital extension.  By construction, $\mathfrak{f}$ is a full extension and hence, $\mathfrak{f}$ is a purely large extension since $\K \oplus \K$ has the corona factorization property.  Since $1_{ \mathbb{B}( \ell^{2} ) } - p$ and $1_{ \mathbb{B}( \ell^{2} ) } - q$ are norm-full, properly infinite projections in $\mathbb{B}( \ell^{2} )$, we have that $1_{ \mathbb{B}( \ell^{2} ) } \oplus 1_{ \mathbb{B}( \ell^{2} ) } - p \oplus q = ( 1_{ \mathbb{B}( \ell^{2} ) } - p ) \oplus ( 1_{ \mathbb{B}( \ell^{2} ) } - q)$ is a norm-full, properly infinite projection in $\mathbb{B}( \ell^{2} ) \oplus \mathbb{B}( \ell^{2} )$.  Moreover, we have that $(1_{ \mathbb{B}( \ell^{2} ) } \oplus 1_{ \mathbb{B}( \ell^{2} ) } - p \oplus q) \F \subseteq \K \oplus \K$.  Therefore, by \cite[Theorem~2.3]{gabe:absorb}, $\mathfrak{f}$ is a nuclear absorbing extension, and hence absorbing since $\C$ is nuclear.  Since $\mathfrak{e}$ and $\mathfrak{f}$ are trivial extensions, we have that $[ \tau_{\mathfrak{e}} ] = [ \tau_{\mathfrak{f}} ] = 0$ in $\kk^{1} ( \C , \K \oplus \K )$.  Thus we have proved the existence of $\mathfrak{f}$.  

Since $\mathfrak{e}$ is not an absorbing extension and $\mathfrak{f}$ is an absorbing extension, we have that $\mathfrak{e}$ is not isomorphic to $\mathfrak{f}$.  Note that
\[
\kk ( \id_{\C} ) \times [ \tau_{\mathfrak{f}} ] = [ \tau_\mathfrak{f} ] = [ \tau_\mathfrak{e} ] = [ \tau_{\mathfrak{e}} ] \times \kk ( \id_{ \K \oplus \K } )
\]  
in $\kk^{1} ( \C , \K \oplus \K )$.  We claim that $\E$ is not isomorphic to $\F$.  Suppose there exists a $*$-isomorphism $\ftn{ \phi }{ \E }{ \F }$.  Let $\pi_{\mathfrak{f}}$ be the canonical surjective $*$-homomorphism from $\F$ to $\C$.  Since $\phi$ and $\pi_{\mathfrak{f}}$ are surjective, we have that $(\pi_{\mathfrak{f}} \circ \phi )( \K \oplus \K )$ is an ideal of $\C$.  So, $(\pi_{\mathfrak{f}} \circ \phi )( \K \oplus \K ) = 0$ or $(\pi_{\mathfrak{f}} \circ \phi) ( \K \oplus \K ) = \C$.  Since $\K \oplus \K$ has exactly four ideals, $0, \K \oplus 0, 0 \oplus \K$, and $\K \oplus \K$, we have that $(\pi_{\mathfrak{f}} \circ \phi) ( \K \oplus \K )$ is either isomorphic to $0$, $\K$, or $\K \oplus \K$.  Hence, $(\pi_{\mathfrak{f}} \circ \phi) ( \K \oplus \K ) = 0$ which implies that $\phi$ maps $\K \oplus \K$ to $\K \oplus \K$.  Similarly, $\phi^{-1}$ maps $\K \oplus \K$ to $\K \oplus \K$.  So, $\phi$ induces an isomorphism of extensions from $\mathfrak{e}$ to $\mathfrak{f}$, which is a contradiction.  Thus, $\E$ is not isomorphic to $\F$.
\end{example}

We correct the error in \cite[Lemma~4.5]{segrer:ccfis} with Proposition~\ref{p: isomorphism from kk} below.  Of particular interest to us in \cite{segrer:ccfis} is the case that the quotient algebra is non-unital.  The main results of \cite{segrer:ccfis} deal with $C^{*}$-algebras that are stable.  Since the quotient of a stable $C^*$-algebra is a stable $C^*$-algebra, we always apply \cite[Lemma~4.5]{segrer:ccfis} to extensions $0 \to \I \to \E \to \A \to 0$ where the quotient algebra $\A$ is a non-unital $C^{*}$-algebra.  So, in this particular case, \cite[Corollary~16]{EK:purelylarge} holds as shown in \cite[Theorem~2.1]{gabe:absorb}.  Thus, using Proposition~\ref{p: isomorphism from kk} in place of \cite[Lemma~4.5]{segrer:ccfis}, the main results of \cite{segrer:ccfis} hold verbatim.

\begin{proposition}\label{p: isomorphism from kk}
For $i =1, 2$, let $\mathfrak{e}_i : 0 \to \I_i \rightarrow \E_i \rightarrow \A_i \to 0$ be a non-unital, full extension of separable, nuclear $C^*$-algebras.  Assume that $\I_i$ is stable and has the corona factorization property.  Suppose there exist $*$-isomorphisms $\phi_2 \colon \A_1 \rightarrow \A_2$ and $\phi_0 \colon \I_1 \rightarrow \I_2$ such that $\mathrm{KK} ( \phi_2 ) \times [ \tau_{\mathfrak{e}_2} ] = [ \tau_{\mathfrak{e}_1} ] \times \mathrm{KK} ( \phi_0 )$.  If 
\begin{itemize}
\item[(i)] $\A_1$ is non-unital or 

\item[(ii)] $\I_1$ is either $\mathbb{K}$ or a purely infinite simple $C^*$-algebra
\end{itemize}
then there exist $*$-isomorphisms $\ftn{ \psi_{1} }{ \E_{1} }{ \E_{2} }$ and $\ftn{ \psi_{0} }{ \I_{1} }{ \I_{2} }$ such that the diagram
\[
\xymatrix{
0 \ar[r] & \I_{1} \ar[r] \ar[d]^{\psi_{0} } & \E_{1} \ar[r] \ar[d]^{\psi_{1}} & \A_{1} \ar[r] \ar[d]^{\phi_{2}}& 0 \\
0 \ar[r] & \I_{2} \ar[r] & \E_{2} \ar[r] & \A_{2} \ar[r] & 0
}
\] 
is commutative and such that $\kk ( \psi_{0} ) = \kk ( \phi_{0} )$.
\end{proposition}

\begin{proof}
Throughout the proof, $\tau_{\mathfrak{e}_{i}}$ will denote the Busby invariant of $\mathfrak{e}_{i}$.  We will also use the fact that a nuclear absorbing extension with quotient algebra nuclear is absorbing.  We will first show that $\mathfrak{e}_{i}$ is an absorbing extension.  Since the extension is full and $\I_{i}$ has the corona factorization property, we have that $\mathfrak{e}_{i}$ is a purely large extension.  Suppose $\A_{1}$ is non-unital.  Since $\A_{1} \cong \A_{2}$, we have that $\A_{2}$ is non-unital.  By \cite[Theorem~2.1]{gabe:absorb}, the extension $\mathfrak{e}_{i}$ is a nuclear absorbing extension, and hence an absorbing extension.  

Suppose $\I_{1}$ is either $\mathbb{K}$ or a purely infinite simple $C^{*}$-algebra.  Since $\I_1 \cong \I_2$, we have that $\I_{2}$ is either $\mathbb{K}$ or a purely infinite simple $C^{*}$-algebra.  So, $\I_{i}$ is the unique non-trivial ideal of $\multialg{\I_{i}}$.  We have two cases to deal with, $\A_{1}$ is non-unital or $\A_{1}$ is unital.  If $\A_{1}$ is non-unital, then so is  $\A_{2}$, and hence $\mathfrak{e}_{i}$ is absorbing from the previous case.  Suppose $\A_{1}$ is unital, then again so is $\A_{2}$.  Recall that there exists a $*$-homomorphism $\ftn{ \sigma_{\mathfrak{e}_{i}}}{\A_{i}}{ \multialg{\I_{i}} }$ such that the diagram 
\[
\xymatrix{
0 \ar[r] & \I_{i} \ar[r] \ar@{=}[d]& \E_{i} \ar[r]^{\pi_{i}} \ar[d]^{\sigma_{e_{i}}} &  \A_{i} \ar[r] \ar[d]^{\tau_{e_{i}}}& 0 \\
0 \ar[r] & \I_{i} \ar[r] & \multialg{\I_{i}} \ar[r]_{\pi} & \corona{\I_{i}} \ar[r] & 0
}
\]
is commutative.  Since $\mathfrak{e}_{i}$ is a non-unital extension, we have that $1_{ \corona{ \I_{i} } } \neq \tau_{ \mathfrak{e}_{i} } ( 1_{\A_{i}} )$.  We claim that there exists a projection $p$ in $\multialg{\I_{i}}$ such that $p$ is not an element of $\I_{i}$ and $\pi ( p ) \leq 1_{ \corona{ \I_{i} } } - \tau_{\mathfrak{e}_{i}} ( 1_{\A_{i}} )$.  Since $\corona{ \I_{i} }$ is a purely infinite, simple $C^{*}$-algebra, there exists a non-zero projection $q$ in $\corona{\I_{i}}$ such that $q \leq   1_{ \corona{ \I_{i} } } - \tau_{e_{i}} ( 1_{\A_{i}} )$ and $q$ is Murray-von Neumann equivalent to $1_{ \corona{ \I_{i} } } = \pi ( 1_{ \multialg{\I_{i}} } )$.  By \cite[Lemma~2.8]{SZ: Multiplier projections}, $q$ lifts to a projection $p$ in $\multialg{\I_{i}}$.  Since $q \neq 0$, we have that $p$ is not an element of $\I_{i}$.  Thus proving the claim.

Since $\I_{i}$ is either $\mathbb{K}$ or a purely infinite simple $C^{*}$-algebra, we have that every projection $e$ in $\multialg{\I_{i}}\setminus \I_{i}$ is norm-full and properly infinite.  Hence, $p$ is a norm-full, properly infinite projection.  Since $\pi (p) \leq 1_{ \corona{ \I_{i} } } - \tau_{\mathfrak{e}_{i}} ( 1_{\A_{i}} )$, we have that $\pi (p) \tau_{ e_{i } } (a) = 0$ for all $a \in \A_{i}$.  Hence, $p \sigma_{\mathfrak{e}_{i}} ( \E_{i} ) \subseteq \I_{i}$.  By \cite[Theorem~2.3]{gabe:absorb}, $\mathfrak{e}_{i}$ is a nuclear absorbing extension and hence an absorbing extension.  Thus we have proved that $\mathfrak{e}_{i}$ is an absorbing extension for all cases.

Let $\mathfrak{f}_{1}$ be the extension obtained by pushing forward the extension $\mathfrak{e}_{1}$ via the $*$-isomorphism $\phi_{0}$ and let $\mathfrak{f}_{2}$ be the extension obtained by pulling-back the extension $\mathfrak{e}_{2}$ via the $*$-isomorphism $\phi_{2}$.  Let $\widetilde{\E}_{1}$ and $\widetilde{\E}_{2}$ be the $C^{*}$-algebras induced by $\mathfrak{f}_{1}$ and $\mathfrak{f}_{2}$ respectively.  Let $\tau_{\mathfrak{f}_{i}}$ be the Busby invariant for the extension $\mathfrak{f}_{i}$.   We claim that $[ \tau_{ \mathfrak{f}_{1}} ] = [ \tau_{ \mathfrak{f}_{2} } ]$ in $\kk^1 ( \A_1 , \I_2 )$.  

By the universal property of the push forward, there exists a $*$-isomorphism $\ftn{\alpha}{ \E_{1} }{ \widetilde{\E}_{1} }$ making the diagram commutative
\[
\xymatrix{0  \ar[r] & \I_{1} \ar[r] \ar[d]^{\phi_{0}} & \E_{1} \ar[r] \ar[d]^{\alpha} & \A_{1} \ar[r] \ar@{=}[d] & 0 \\
0 \ar[r] & \I_{2} \ar[r] & \widetilde{\E}_{1} \ar[r] & \A_{1} \ar[r] & 0.
}
\]
Using the universal property of the pull-back, there exists a $*$-isomorphism $\ftn{ \beta }{ \widetilde{\E}_{2} }{ \E_{2} }$ making the diagram commutative
\[
\xymatrix{0  \ar[r] & \I_{2} \ar[r] \ar@{=}[d] & \widetilde{\E}_{2} \ar[r] \ar[d]^{\beta} & \A_{1} \ar[r] \ar[d]^{\phi_{2}} & 0 \\
0 \ar[r] & \I_{2} \ar[r] & \E_{2} \ar[r] & \A_{2} \ar[r] & 0.
}
\]
By \cite[Proposition~1.1]{MK: Extensions Kirchberg},
\[
 [ \tau_{ \mathfrak{e}_{1}} ] \times \kk ( \phi_{0} ) =    [ \tau_{ \mathfrak{f}_{1}} ]
 \]
 and 
\[
[ \tau_{ \mathfrak{f}_{2}} ]  = \kk ( \phi_{2} ) \times  [ \tau_{ \mathfrak{e}_{2}} ].
\]
Thus, $[ \tau_{ \mathfrak{f}_{1}} ] = [ \tau_{ \mathfrak{e}_{1}} ] \times \kk ( \phi_{0} ) = \kk ( \phi_{2} ) \times  [ \tau_{ \mathfrak{e}_{2}} ] = [ \tau_{ \mathfrak{f}_{2} } ]$ in $\kk^{1} ( \A_{1} , \I_{2} )$, proving the claim that $[ \tau_{ \mathfrak{f}_{1}} ] = [ \tau_{ \mathfrak{f}_{2} } ]$ in $\kk^1 ( \A_1 , \I_2 )$

Since $\A_{1}$ is a nuclear, separable $C^{*}$-algebra and since $[ \tau_{ \mathfrak{f}_{1}} ] = [ \tau_{ \mathfrak{f}_{2} } ]$ in $\kk^1 ( \A_1 , \I_2 )$, there are trivial extensions $\ftn{ \sigma_{1}, \sigma_{2} }{ \A_{1}  }{\corona{\I_{2} }}$ and there exists a unitary $v \in \multialg{\I_{2} }$ such that $\mathrm{Ad}(\pi(v)) \left( \tau_{ \mathfrak{f}_{1} } \oplus \sigma_{1} \right) = \tau_{ \mathfrak{f}_{2}} \oplus \sigma_{2}$, where $\pi$ is the canonical surjective $*$-homomorphism from $\multialg{\I_{2} }$ onto $\corona{ \I_{2} }$.  Since $\mathfrak{e}_{i}$ is an absorbing extension, we have that $\mathfrak{f}_{i}$ is an absorbing extension.  Hence, there exists a unitary $v_{i} \in \multialg{ \I_{2} }$ such that $\mathrm{Ad} ( \pi ( v_{i} ) ) \circ \left( \tau_{ \mathfrak{f}_{i} } \oplus \sigma_{i} \right) = \tau_{ \mathfrak{f}_{i} }$.  Set $U = v_{2} v v_{1}^{*}$.  A computation shows that $\mathrm{Ad} ( \pi(U) ) \circ \tau_{ \mathfrak{f}_{1}} = \tau_{ \mathfrak{f}_{2}}$.  Therefore, $\mathrm{Ad} ( U)$ induces $*$-isomorphisms $\ftn{ \lambda_{0} }{ \I_{2} }{ \I_{2} }$ and $\ftn{ \lambda_{1} }{ \widetilde{\E}_{1} }{ \widetilde{\E}_{2} }$ such that the diagram
\[
\xymatrix{
0 \ar[r] & \I_{2} \ar[r] \ar[d]^{\lambda_{0}} & \widetilde{\E}_{1} \ar[r] \ar[d]^{\lambda_{1}} & \A_{1} \ar[r] \ar@{=}[d]& 0\\
0 \ar[r] & \I_{2} \ar[r] & \widetilde{\E}_{2} \ar[r] & \A_{1} \ar[r] & 0 
}
\]
is commutative and $\kk ( \lambda_{0} ) = \kk ( \id_{\I_{2}} )$.

Set $\psi_{0} = \lambda_{0} \circ \phi_{0}$ and $\psi_{1} = \beta \circ \lambda_{1} \circ \alpha$.  Then $\psi_{0}$ and $\psi_{1}$ satisfies the desired properties.
\end{proof}

We end by commenting on other results by the authors that relied on \cite[Corollary~16]{EK:purelylarge} and/or \cite[Lemma~4.5]{segrer:ccfis}.  

\begin{observation}
As proved in \cite[Theorem~2.1]{gabe:absorb} that the last part of \cite[Corollary~16]{EK:purelylarge} holds.  More precisely, for an extension $\mathfrak{e}\colon 0 \to \I \to \E \to \A \to 0$ with $\A$ non-unital, $\mathfrak{e}$ is nuclear absorbing if and only if $\mathfrak{e}$ is purely large.  Consequently, \cite[Corollary~16]{EK:purelylarge} holds when dealing with extensions of stable $C^{*}$-algebras since a quotient of a stable $C^*$-algebra is stable.  Therefore, the results of \cite{ERRshift} and \cite{ERRlinear} hold since both articles consider extensions of stable $C^{*}$-algebras.
\end{observation} 
 
 \begin{observation} 
In \cite[Theorem~2.6]{rr_cexpure}, the second and third named author used \cite[Corollary~16]{EK:purelylarge} for extensions $0 \to \I \to \E \to \A \to 0$ where $\I$ is a purely infinite simple $C^{*}$-algebra.  Thus, using Proposition~\ref{p: isomorphism from kk} in place of \cite[Corollary~16]{EK:purelylarge}, we have that \cite[Theorem~2.6]{rr_cexpure} holds as stated.
\end{observation}

\begin{observation}
In \cite[Lemma~6.13(a)]{ERS:amplified}, the first and third named author with Adam S{\o}rensen proved a similar result as \cite[Lemma~4.5]{segrer:ccfis} using \cite[Corollary~16]{EK:purelylarge}.  Although, \cite[Lemma~6.13]{ERS:amplified} is incorrect as stated, it was only applied in \cite[Theorem~6.17]{ERS:amplified} for extensions $0 \to \I \to \E \to \A \to 0$ where $\I$ is $\K$ or a purely infinite simple $C^{*}$-algebra.  Therefore, replacing \cite[Lemma~6.13(a)]{ERS:amplified} with Proposition~\ref{p: isomorphism from kk}, \cite[Theorem~6.17]{ERS:amplified} holds as stated. 
\end{observation}

\begin{observation}
In \cite[Theorem~4.9]{err:fullext}, the authors give a complete classification of all graph $C^{*}$-algebras with exactly one non-trivial ideal.  This result relied on \cite[Lemma~4.5]{segrer:ccfis}.  Using Proposition~\ref{p: isomorphism from kk}, \cite[Theorem~4.9]{err:fullext} is false in exactly one case.  It is false in general for the case of non-unital graph $C^{*}$-algebras $C^{*} (E)$ with exactly one non-trivial ideal $\I$ with $\I$ an AF-algebra and $C^{*} (E) / \I$ a unital purely infinite simple $C^{*}$-algebra.  Using \cite[Example~1.1]{gabe:absorb} as inspiration, one can construct two non-isomorphic, non-unital graph $C^{*}$-algebras $C^{*} (E_{1} )$ and $C^{*} (E_{2} )$ such that each $C^{*} (E_{i} )$ has exactly one non-trivial ideal $\I_{i}$, $C^{*} (E_{i} ) / \I_{i}$ is a unital, purely infinite, simple $C^{*}$-algebra, $\I_{i}$ is an AF-algebra, and $K_{\mathrm{six}} ( C^{*} (E_{1} ) ; \I_{1} ) \cong K_{\mathrm{six}} ( C^{*} (E_{2} ) ; \I_{2} )$ with an isomorphism that is a scale and order isomorphism.
\end{observation}

\def\cprime{$'$}

\end{document}